\documentclass[12pt,a4paper,oneside]{article}
\usepackage{amsfonts, amsmath, amssymb,latexsym, amsthm}
\usepackage{epsfig}
\usepackage{color}
\usepackage[all]{xy}
\usepackage{array}
\usepackage{enumitem}

\setlist[itemize,1]{leftmargin=15pt}
\setlist[enumerate,1]{leftmargin=15pt}

\usepackage[latin1]{inputenc}

\newtheorem{thm}{Theorem}[section]
\newtheorem{prop}[thm]{Proposition}

\newtheorem{lemma}[thm]{Lemma}

\newtheorem{cor}[thm]{Corollary}
\newtheorem{rmk}[thm]{Remark}

\usepackage[
colorlinks]{hyperref}

\usepackage{memhfixc} 
\usepackage
{hypcap} 
\hypersetup{
    bookmarksnumbered,
    pdfstartview={FitH},
    citecolor={blue},
    linkcolor={green},
    urlcolor={red},
    pdfpagemode={UseOutlines}
}

\usepackage{fancyhdr}

\usepackage{float}
\restylefloat{table}

\title{On degree three curves in $C^{(2)}$ with positive self-intersection}
\author{Meritxell S\'aez}
\date{}
\begin{document}
\bibstyle{plain}

\maketitle

\begin{abstract}
In this paper we study degree three curves in $C^{(2)}$ with positive self-intersection defined by the action of a spherical triangular group in an auxiliary curve.
{\parskip7pt

\parskip7pt
\noindent \textbf{Keywords}: Symmetric product, curve, irregular surface, curves in surfaces.}
\let\thefootnote\relax\footnote{The author has been partially supported by the Proyecto de Investigaci\'on MTM2012-38122-C03-02.}
\end{abstract}

\section{Introduction}

In this paper we study degree three curves in $C^{(2)}$, that is, curves $\tilde{B}\subset C^{(2)}$ with $\tilde{B}\cdot C_P=3$. We continue the work we began in \cite{MS2} and \cite{MS3} on the study of curves in $C^{(2)}$ with special attention to those with positive self-intersection.

A fundamental tool for this study is the main theorem in \cite{MS2} where curves in $C^{(2)}$ with irreducible preimage in $C\times C$ and degree $d$ are characterized. In \cite{MS3} we used this result to completely classify degree two curves with positive self-intersection.

For $d=3$ the main result in \cite{MS2} translates into

\begin{thm}\label{charcurv}
Let $B$ be an irreducible smooth curve such that there are no non-trivial morphisms $B\rightarrow C$. A morphism of degree one from the curve $B$ to the surface $C^{(2)}$ exists, with image $\tilde{B}$ of degree $3$ if, and only if, there exists a smooth irreducible curve $D$ and a diagram 
\[
\xymatrix{
D \ar[d]_{(3:1)}\ar[r]^{(2:1)}& B\\
C &
}
\]
which does not complete.
\end{thm} 

That is, that does not exist a curve $H$ and maps such that we obtain a commutative diagram
\[
\xymatrix{
D \ar[d]_{(3:1)}\ar[r]^{(2:1)}& B \ar@{.>}[d]^{(3:1)}\\
C \ar@{.>}[r]_{(2:1)}& H.
}
\]

In \cite[Question 8.6]{MPP2} the authors wonder if there exists a curve $B$ in a surface $S$ with $q(S)<p_a(B)<2q(S)-1$ (the Brill-Noether range) and $B^2>0$. This question relates also with the existence of a curve of genus $q<p_a(C)<2q-1$ that generate an abelian variety of dimension $q$ (see \cite{Pi2}). In \cite{MS3} we saw that for large $g(C)$ it is more likely that such a curve has low degree, hence motivating the study of low degree curves in the symmetric square. In \cite{MS3} we studied the degree two case in detail. In this paper we consider some degree three cases. We find no further examples of curves with positive self-intersection and arithmetic genus in the Brill-Noether range, even when considering their preimages in $C\times C$, as was the case in \cite{MS3}.

First, we prove that the preimage of $\tilde{B}$  by $\pi_C:C\times C \rightarrow C^{(2)}$ is always irreducible:

\begin{prop}\label{trirred}
Let $\tilde{B}\subset C^{(2)}$ be a degree $3$ curve. Then $\pi_C^*(\tilde{B})$ is irreducible.
\end{prop}

Therefore, from Theorem \ref{charcurv} we deduce that all curves of degree $3$ are defined by a diagram of curves that does not complete. Since not all degree $3$ morphisms are Galois, not all diagrams come from the action of a group in a curve $D$ as happened in the degree two case. We are going to study some special cases in the Galois situation, specifically those where $D$ is a curve with two automorphisms: $i$ of order $2$ and $\alpha$ of order $3$ such that $\langle i, \alpha \rangle = S_4, A_4\ \mathrm{or}\ A_5$, the so called spherical triangle groups, because of their simple and well known structure. The groups $S_3$ and $D_3$ are also spherical triangle groups, nevertheless, since they are such that $|D_3|=|S_3|=3\cdot 2$, the diagram obtained would complete (see \cite[Proposition 1.2]{MS3}). Moreover, we observe that since there are an infinity of groups of finite order generated by an involution and an order three element (see \cite{Mi2}), a complete study even only of the Galois case presents a great complexity using this approach. 

Given an automorphism $\beta$ we denote by $\nu(\beta)$ the number of points fixed by $\beta$. We find that

\begin{thm}\label{summarytr}
Let $D$ be a curve with the action of two automorphisms: $i$ of order two and $\alpha$ of order three. Assume that $\langle i, \alpha \rangle=A_4, S_4\ \textrm{or}\ A_5$. Let $C=D/\langle \alpha \rangle$ and $B=D/\langle i \rangle$. Then, there exists a curve $\tilde{B}\subset C^{(2)}$ of degree three, with normalization $B$ (as in Theorem \ref{charcurv}).

Moreover, if we denote $\pi_C^*(\tilde{B})=:\tilde{D}$, then $\tilde{D}$ has normalization $D$, $\tilde{B}$ has $\frac{1}{2}\nu(i\alpha^2i\alpha)+\frac{1}{2}(\nu((i\alpha)^2)-\nu(i\alpha))$ nodal singularities and $\tilde{D}$ has $\nu(i\alpha^2i\alpha)+\nu((i\alpha)^2)$ nodal singularities.    
\end{thm}

We analyse next which of these curves $\tilde{B}$ have positive self-intersection in Section \ref{deg3higher}. 

In particular, the curves given by the action of $A_4$ are described in Tables 
\ref{A4t3} 
 and \ref{A4t10}. The curves given by the action of $S_4$ are described in Tables \ref{S4t1} and \ref{S4t2} and the curves given by the action of $A_5$ are described in Table \ref{A5t1}.  

Analysing the genus of the curves $\tilde{B}$ and $C$ in the different cases we deduce the following corollary:

\begin{cor}
The curves $\tilde{B}$ in $C^{(2)}$ of degree $3$ defined by the action of $G=A_4,S_4,A_5$ on a curve $D$ with arithmetic genus in the Brill-Noether range have non-positive self-intersection.
\end{cor}

The outline of the paper is as follows. In Section \ref{Backgroups} we recall some basic facts about group actions on curves that we will need during the main part of the paper. In Section \ref{triang} we study the singularities of curves of degree three defined by the action of a spherical triangular group. In Section \ref{deg3higher} we describe all such curves with positive self-intersection.

\noindent \textbf{Acknowledgments.} The most sincere gratitude to Miguel Angel Barja and Joan Carles Naranjo for the multiple discussions and the amount of time devoted to the development of this article. And finally to the Universitat de Barcelona for the research grant and their hospitality afterwards.

\textbf{Notation:} We work over the complex numbers. By curve we mean a complex projective reduced algebraic curve. Let $C$ be a smooth curve of genus $g\geq 2$, we put $C^{(2)}$ for its $2$nd symmetric product. We denote by $\pi_C:C \times C \rightarrow C^{(2)}$ the natural map, and $C_P \subset C^{(2)}$ a coordinate curve with base point $P\in C$. We denote by $p_a(C) = h^1(C,\mathcal{O}_C)$ the arithmetic genus and when $C$ is smooth by $g(C)= h^0(C, \omega_{C})$ the geometric genus (or topological genus). We will call node an ordinary singularity of order two.

For $\alpha \in \mathrm{Aut}(C)$, we denote by $\nu(\alpha)$ the number of points fixed by $\alpha$. We put $\Gamma_{\alpha}$ for the curve in $C\times C$ given by the graph of $\alpha$, that is, $\Gamma_{\alpha}=\{ (x, \alpha(x)),\ x\in C\}$.

\section{Background on group actions}\label{Backgroups}

We recall here some basic facts about group actions on curves. 

Let $C$ be a curve and let $G\subset \mathrm{Aut}(C)$ be a finite subgroup. For $P\in C$, set $G_P=\{g\in G\ |\ g(P)=P\}$ the \textbf{stabilizer} of $P$.

\begin{prop}\label{cyc}(\cite[III.7.7]{Far})
Assume $g(C)\geq 2$. Then $G_P$ is a cyclic subgroup of $\mathrm{Aut}(C)$.
\end{prop}

In particular, if $\alpha,\beta\in Aut(C)$ are not powers of a common $\gamma\in Aut(C)$, they have no common fixed point.

Given $\alpha\in Aut(C)$, its graph $\Gamma_{\alpha}$ lies in $C\times C$ and is isomorphic to $C$. With a local computation one can see that

\begin{prop}\label{diagtrans}
The diagonal in $C\times C$ cuts the graph of an automorphism transversally.
\end{prop}

\begin{cor}(\cite{MS3})\label{grtrans}
Let $\alpha$ and $\beta$ be two automorphisms of a curve $C$. If $\alpha^{-1}\beta\neq 1$, then the graphs of $\alpha$ and $\beta$ in $C \times C$ intersect transversally and moreover, $\Gamma_{\alpha}\cdot \Gamma_{\beta}$ equals the number of points fixed by the automorphism $\alpha^{-1}\beta$, that is, $\nu(\alpha^{-1}\beta)$.
\end{cor}

\begin{lemma}(\cite{MS3})\label{orbit}
Let $G$ be a finite group of order $n$ acting on a curve $C$. Given a point $P\in C$, let $\alpha$ be a generator of $G_P$. Then we have that
\[
n=|G_P|\cdot |\{\textrm{conjugates of }G_P\}|\cdot |\{\textrm{points fixed by }\alpha\textrm{ in } \mathcal{O}_G(P)\}|.
\]
\end{lemma}

\begin{thm}[Riemann's Existence Theorem]\label{rexthm}
The group $G$ acts on a curve of genus $g$, with branching type $(g'; m_1, \dots, m_r)$ if and only if the Riemann-Hurwitz formula is satisfied and $G$ has a $(g'; m_1, \dots, m_r)$ generating vector.
\end{thm}

Where a $(g'; m_1, \dots, m_r)$ generating vector (or $G$-Hurwitz vector) is a $2g'+r$-tuple 
\[
(a_1, b_1, \dots, a_{g'},b_{g'}; c_1, \dots, c_r)
\] 
of elements of $G$ generating the group and such that $o(c_i)=m_i$ and 
\[
\prod\limits_{j=1}^{g'} [a_i,b_i] \prod\limits_{i=1}^{r}c_i=1.
\]
We call this last condition the \textbf{product one condition}. 

We remark that Riemann's Existence theorem is not a constructive result. It states the existence of such a curve, nevertheless it gives no further information about it.

\section{On degree three curves}\label{triang}

Next, we study some general properties of degree three curves in $C^{(2)}$ defined by a spherical triangular group. First of all, we prove a general property of all degree three curves, that is, we see that $\pi^{*}_C(\tilde{B})$ is always irreducible.

\begin{proof}[Proof of Proposition \ref{trirred}.] Let $B$ be the normalization of $\tilde{B}$. If $\pi_C^*(\tilde{B})$ were reducible, then $\pi_B^*(\tilde{B})=B_1+B_2$ with $B_1$ and $B_2$ two divisors with normalization $B$. Since we have a morphism from $\pi_C^*(\tilde{B})$ to $C$ of degree $3$, one, let us say $B_1$, would have a degree one morphism to $C$ and $B_2$ would have a degree two morphism to $C$. But then, on the one hand, $B$ and $C$ are isomorphic and on the other hand there is a degree two morphism from $B$ to $C$, a contradiction since we are assuming that $g(C)\geq 2$. 
\end{proof}

Hence, a degree three curve $\tilde{B}\subset C^{(2)}$, with normalization $B$, has preimage by $\pi_C$ an irreducible curve $\tilde{D}:=\pi_{C}^*(\tilde{B})$, which has normalization $D$. Regarding Theorem \ref{charcurv}, there exists a diagram of curves which does not complete defined by $\tilde{B}\subset C^{(2)}$.

As mentioned in the Introduction, we are going to study some special cases in the Galois situation, specifically those where $D$ is a curve with two automorphisms: $i$ of order $2$ and $\alpha$ of order $3$ such that they do not commute, giving a diagram \vspace{-10pt}
\begin{align}\label{diag}
\xymatrix{
D \ar[r]^{f} \ar[d]^{g}& B=D/\langle i\rangle  \\
C=D/\langle \alpha\rangle & 
} &
\begin{array}{l}
\\ \\
\text{ with } 
\langle i,\, \alpha\rangle =A_4,S_4,A_5.
\end{array}
\end{align}
We consider the curve $D$ embedded in $D\times D$ as the set of points $\{(x,i(x))\}$. We are going to study the singularities of $\tilde{D}\subset C\times C$ and $\tilde{B}\subset C^{(2)}$ to prove Theorem \ref{summarytr}.

Let $R\subset D^{(2)}$ be the divisor defined as 
\begin{equation}\label{defR}
R=\{x+y\ |\ g(x)=g(y)\}.
\end{equation} 
Since $g$ is the quotient by the action of $\alpha$, we obtain that 
\[
\pi_{D}^*R=\overline{\{(x,y) \ |\ g(x)=g(y),\ x\neq y\}}=
\{(x,\alpha(x))\}\, +\, \{(x,\alpha^2(x))\}=\Gamma_{\alpha}+\Gamma_{\alpha^2}.
\]

The points in $\pi_{D}^*R\cap D$ are pairs of different points in $D \subset D \times D$ with the same image in $\tilde{D}$, so their images by $g\times g$ are singularities of $\tilde{D}$. We are going to see that their images in $\tilde{B}$ are smooth points.

\begin{lemma}\label{trsmooth}
The image in $\tilde{B}$ by $\pi_C|_{\tilde{D}}$ of a point $(g\times g)(x,i(x))$ with $i\alpha(x)=x$ or $i\alpha^2(x)=x$ is a smooth point where the curve $\tilde{B}$ is tangent to the diagonal. 
\end{lemma}

\begin{proof} First, we study the singular points in $\tilde{D}$ of the form $(g\times g)(x,i(x))$. 

Consider the morphism $g\times g: D \times D \rightarrow C\times C$. It is Galois with group $\langle 1\times \alpha, \alpha \times 1\rangle$. We consider $D=\{(x, i(x))\}$ and all its images by the elements of that group, that is, all preimages of $\tilde{D}$ by $g\times g$. 

Since $g \times g$ is Galois, each preimage curve is the graph of an automorphism in $G\subset Aut(D)$ and hence they intersect pairwise transversally (Corollary \ref{grtrans}).

We consider first singular points in $\tilde{D}$ corresponding to $(x, i(x))$ with $i\alpha(x)= x$. Each of these singular points has as preimages one point $(x, i(x))\in D\cap (\alpha \times \alpha^2)D$ and one point $(i(x), x)\in D\cap (\alpha^2\times \alpha)D$. 
Since $g\times g$ is not ramified in these points, and $D$ and its image by $\alpha^2\times \alpha$ are transversal, we deduce that $\tilde{D}$ is transversal on the image, and therefore, the images are nodes in $\tilde{D}$.

Since the points $(x,i(x))$ and $(i(x),x)=(\alpha(x), i\alpha(x))$ have the same image by the morphism $\pi_{D}$, there is only one point for each of these singularities in $B$, the normalization of $\tilde{B}$. Then, doing a local computation we deduce that $\tilde{B}$ is smooth and tangent to the diagonal in $C{(2)}$ in each of these points. 

Finally, since given a point $x$ with $i\alpha(x)=x$ the images of $(x,i(x))$ and $(i(x), x)$ are equal, we have also proved the lemma for those $x\in D$ with $i\alpha^2(x)=x$, because in that case $i(x)$ is a point fixed by $i\alpha$. 
\end{proof}

Now, we study the other singularities of  $\tilde{D}$ and $\tilde{B}$.

\begin{prop}\label{singBtr}
\[
|\mathrm{Sing}\, \tilde{B}|= \frac{1}{2}\nu(i\alpha^2i\alpha)+\frac{1}{2}(\nu((i\alpha)^2)-\nu(i\alpha)).
\]
\end{prop}

\begin{proof}
First, we want to know when two different points in $D$ have the same image in $\tilde{D}$. We remind that $D\rightarrow \tilde{D}$ is the normalization map.

Let $(x,y)$ (with $i(x)=y$) and $(z,t)$ (with $i(z)=t$) be two different points with the same image by $g\times g$, that is, such that $\alpha^k(x)=z$ and $\alpha^r(y)=t$ for certain $k,r\in\{1,2\}$. Given such two pairs, we obtain that
\[
\begin{array}{ccccl}
x=&i(y)=&i\alpha^{3-r}(t)=& i\alpha^{3-r}i(z)=&i\alpha^{3-r}i\alpha^{k}(x),
\end{array}
\]
and similarly $y=i\alpha^{3-k}i\alpha^{r}(y)$, $z=i\alpha^{r}i\alpha^{3-k}(z)$ and $t=i\alpha^{k}i\alpha^{3-r}(t)$, i.e. they are points fixed by certain automorphisms.

We have four possibilities for $k$ and $r$ that can be gathered in two cases:

\vspace{5pt}
\textbf{Case A:} if $k=r\in \{1,2\}$, then the two points in each involution pair are fixed by the same automorphism, for instance, $x$ and $y$ are fixed by $i\alpha^2 i \alpha$ and the points $z$ and $t$ are fixed by $i\alpha i \alpha^2$.

Assume that it is the case, that is, $\boldsymbol{k=r}$. 

Let $x\in D$ be such that $i\alpha^2 i \alpha(x)=x$, that is, $i\alpha(x)=\alpha i(x)$, and take $y:=i(x)\neq x$. We denote by
\[
z_1:=\alpha(x) \hspace{7pt} t_1:=\alpha(y) \hspace{15pt} z_2:=\alpha^2(x) \hspace{7pt} t_2:= \alpha^2(y).
\]
If we consider $z_0:=x$ and $t_0:=y$ then the pairs $(z_n, t_m)$ form a fiber of the morphism $g\times g$. We notice that $i(z_1)=t_1$, so we obtain that $(z_1,t_1)\in D \subset D \times D$. 

We claim that $i(z_2)\neq t_2$. Otherwise, $i(z_2)=t_2$, so $i\alpha^2(x)=\alpha^2i(x)$ and hence $i\alpha i \alpha^2(x)=x$. This would imply that there exists a cyclic group containing both $i\alpha^2i\alpha$ and $i\alpha i\alpha^2$ (see Proposition \ref{cyc}). With a detailed analysis of the multiplication tables for our groups $\langle i,\alpha\rangle=A_4,S_4,A_5$ we see that we have reached a contradiction.

Therefore, we have two different pairs of points on $D$ with image in $\tilde{D}$ pairwise equal, that is, two singularities with two branches. We notice that $x$ and $y$ are both fixed by $i\alpha^2i\alpha$, and hence there are $\nu(i\alpha^2i\alpha)$ singularities in $\tilde{D}$ coming from this kind of points. Notice that the image of $\{(x,y),(z_1,t_1)\}$ and that of $\{(y,x),(t_1,z_1)\}$ will be two different singularities in $\tilde{D}$ with the same image in $\tilde{B}$, so they give $\frac{1}{2}\nu(i\alpha^2i\alpha)$ singularities in $\tilde{B}$ because they are not on the branch locus of $\pi_C$.

\vspace{5pt}  
\textbf{Case B:}  if $k\neq r$, $\{k,r\}= \{1,2\}$, then one of the two points in an involution pair is fixed by $i\alpha i\alpha$ and the other by $i\alpha^2 i\alpha^2$.

Assume that it is the case, that is, $\boldsymbol{k\neq r}$. 

Let $x\in D$ be a point such that $i\alpha i \alpha(x)=x$ with $i\alpha(x)\neq x$. We notice that this is only possible when $\langle i,\alpha\rangle =S_4 $ because in the other two cases the order of $i\alpha$ is prime, and hence the points fixed by it and its square are the same. Those points with $i\alpha(x)= x$ have been already considered in Lemma \ref{trsmooth} where we have seen that their images in $\tilde{B}$ are smooth points. 

With an analysis similar to the previous one we deduce that this kind of points give $\frac{1}{2}\left(\nu((i\alpha)^2)-\nu(i\alpha)\right)$ singularities in $\tilde{B}$.

\vspace{5pt}
Next, we will see that there are no other singularities in $\tilde{B}$. For a point in $B$ outside the ramification locus of $g^{(2)}$, its image will be a singularity only if there is another point with the same image, because $g^{(2)}$ is a local homeomorphism around it. This is the case we have just studied. Then, it only remains to consider the ramification points of $g^{(2)}$. 

In $B\subset D^{(2)}$ there are two types of points where $g^{(2)}$ ramifies: those in $R$ (see (\ref{defR})) and those in $D_x$ with $x$ a ramification point of $g$. We have seen in Lemma \ref{trsmooth} that the image of a point in $B\cap R$ is always smooth, so it remains only to study those in $D_x$ for $x\in \mathrm{Ram}(g)$.    

To do this, we study the intersection of $\tilde{B}$ with a coordinate curve $C_P$, with $P\in \mathrm{Branch}(g)$. We remind that $C_P\cdot \tilde{B}=3$.

Let $P\in \mathrm{Branch}(g)$ i.e. $\exists !\ x $ such that $g(x)=P$, that is, $x$ is a point fixed by $\alpha$. Let $y:=i(x)\neq x$, then, $g^{(2)}(x+y)=P+g(y)=P+Q$. 

We know that $C_P$ intersects $\tilde{B}$ in a single point $P+Q$ with multiplicity three. We want to know how $C_Q$ intersects $\tilde{B}$ to prove that it is a smooth point. We distinguish two cases:

First, if $y\in \mathrm{Ram}(g)$, that is $\alpha(y)=y$ i.e. $i\alpha i(x)=x$, then, since we are assuming that $\alpha(x)=x$, this would imply that $\langle \alpha, i\alpha i\rangle $ is contained in a cyclic group, which is not possible. 

Second, if $y\notin \mathrm{Ram}(g)$ then there exist $t$ and $z$ such that $\alpha(y)=t$ and $\alpha^2(y)=z$. If $\exists k$ such that $\alpha^{k}i(t)=i(z)$, then there would be a point in the intersection of $C_Q$ and $\tilde{B}$ with multiplicity greater than $1$, otherwise there would be two different points in this intersection: $g^{(2)}(t+i(t))$ and $g^{(2)}(z+i(z))$. In any case, these points do not belong to $C_P$, and hence, in $P+Q$ the intersection multiplicity of $C_Q$ and $\tilde{B}$ is one. Therefore, the curve $\tilde{B}$ is smooth at $P+Q$. Hence, there are no more singular points.

\end{proof}

\begin{cor}\label{singGammatr}
\[
|\mathrm{Sing}\,\tilde{D}|=\nu(i\alpha^2i\alpha)+\nu((i\alpha)^2).
\]
\end{cor}

\begin{proof}
Analysing the preimages of the singularities of $\tilde{B}$ and the possible tangencies of $\tilde{D}$ and the diagonal we see that the curve $\tilde{D}$ does not have more singularities than those considered during the proof of Proposition \ref{singBtr}.
\end{proof}

\vspace{5pt}
Now, we study which kind of singularities they are. 

\begin{prop}\label{nodestr}
All singularities in $\tilde{D}$ and $\tilde{B}$ are nodes.
\end{prop}

\begin{proof}
We begin studying the singularities on $\tilde{D}$ and later their image by $\pi_C$.

As we have seen in the proof of Lemma \ref{trsmooth}, the preimage of $D$ by $g\times g$ consists of the graphs of the elements in $\langle 1\times \alpha, \alpha \times 1 \rangle$. These divisors intersect transversally and $D$ is also transversal in the image of these points.

Taking the intersections of $D$ with its images by elements of the group $\langle 1\times \alpha, \alpha \times 1 \rangle$, we recover the cases in the proof of Proposition \ref{singBtr} of possible singularities, that in this language are
\[
\begin{array}{cc}
\left.{\begin{array}{ccc}
i\alpha^2i\alpha(x)=x & \Leftrightarrow i\alpha(x)=\alpha i (x)& \leftrightarrow D \cap (\alpha \times \alpha)D \\
i\alpha i\alpha^2(x)=x & \Leftrightarrow i\alpha^2(x)=\alpha^2 i (x) & \leftrightarrow D \cap (\alpha^2\times \alpha^2)D \end{array}}\right\} & {\begin{array}{c}\alpha^2 \times \alpha^2\\ \textrm{permutes them.} \end{array}}\\[15pt]
\left.{\begin{array}{ccc}
i\alpha^2i\alpha^2(x)=x & \Leftrightarrow i\alpha^2(x)=\alpha i (x) & \leftrightarrow D \cap (\alpha^2\times \alpha)D \\
i\alpha i\alpha(x)=x & \Leftrightarrow i\alpha(x)=\alpha^2 i (x) & \leftrightarrow D \cap (\alpha \times \alpha^2)D \end{array}}\right\} & {\begin{array}{c}\alpha^2 \times  \alpha\\ \textrm{permutes them.} \end{array}}
\end{array}
\]
The first two correspond to Case A and the last two correspond to Case B. 

Notice that, those singular points $(g(x_0), g(i(x_0)))$ with $i\alpha(x_0)\neq x_0$ are not on the diagonal of $C\times C$, and hence $\pi_C$ does not ramify on them. Therefore, their images on $C^{(2)}$ are also nodes. 

Those singular points $(g(x_0), g(i(x_0)))$ with $i\alpha(x_0)= x_0$ are on the diagonal of $C\times C$, and hence $\pi_C$ ramifies on them. We have already seen in Lemma \ref{trsmooth} that their images in $\tilde{B}$ are smooth points. 
\end{proof}   

\vspace{5pt}
Therefore, by Propositions \ref{singBtr} and \ref{nodestr} we obtain that

\begin{cor}\label{genBtr}
\[
p_a(\tilde{B})-g(B)=\frac{1}{2}\nu(i\alpha^2i\alpha)+\frac{1}{2}(\nu((i\alpha)^2)-\nu(i\alpha)).
\]
\end{cor}

With Propositions \ref{singBtr} and \ref{nodestr}, and Corollary \ref{singGammatr} we have proven Theorem \ref{summarytr}.

\begin{rmk}\label{selfinttr}
Moreover, by \cite[Lemma 2.1]{MS3}, we deduce that:
\[
\tilde{B}^2 = g(D)-1-3(2g-2)+ \nu(i\alpha^2i\alpha) + \nu((i\alpha)^2).
\]
\end{rmk}

\section{Positive self-intersection curves}\label{deg3higher}

Now, we consider curves with positive self-intersection, that is, we consider $\tilde{B}^2>0$. We will describe all generating vectors of $G$ with $G=A_4,\ S_4,\ A_5$ that give a non-completing diagram of morphisms of curves characterizing a curve $\tilde{B}\subset C^{(2)}$ with $\tilde{B}^2>0$ and $g(C)\geq 2$.

We are going to consider separately each group $G=A_4,S_4, A_5$. We begin with a numerical analysis of our hypothesis and later, for those values compatible with the hypothesis, we give (or prove that it does not exist) a generating vector defining a curve $D$ with an action of $G$ and the prescribed ramification. To simplify the notation, we describe the generating vector of $G$ giving a product one relation of elements of $G$. Each generator is written in square brackets $[\cdot]$ and the exponent of the brackets denote the number copies in the vector. We prove in the notes after the tables that the elements taken generate the whole group when it is not absolutely clear. 

We have by hypothesis that the curves $C$ and $B$ lay in a diagram as \eqref{diag}. We denote by 
\[
\begin{array}{ccccc}
b=g(B) & g=g(C) & h=g(D)&&\\[3pt]
s=\nu(\alpha) & t=\nu(i) & r=\nu(i\alpha) & r+k=\nu((i\alpha)^2) & e=\nu(i\alpha^2i\alpha).
\end{array}
\]

With this notation, the equality in Remark \ref{selfinttr} translates into 
\[
\tilde{B}^2=h-1-3(2g-2)+e+k+r.
\]

First, we are going to use our hypothesis to give some restrictions for the possible values of $b,g,h,s,t,r,k$ and $e$.
\begin{itemize}
\item  By the Riemann-Hurwitz formula for the morphism $D \rightarrow C$ we obtain
\begin{equation}\label{3a}
g=\frac{h+2-s}{3}.
\end{equation}
\item Since $\tilde{B}^2>0$, by (\ref{3a}) we obtain that
\begin{equation}\label{3b}
 h\leq 2s+e+k+r.
\end{equation}
\item By the Riemann-Hurwitz formula for the morphism $D \rightarrow B$ we obtain that $ b=\dfrac{2h+2-t}{4}$. Therefore, by Propositions \ref{singBtr} and \ref{nodestr} we deduce that
\begin{equation}\label{3c}
p_a(\tilde{B})=\frac{2h+2-t}{4}+\frac{1}{2}(e+k).
\end{equation}
\item By (\ref{3a}) and (\ref{3c}) the necessary inequality $g<p_a(\tilde{B})$ translates into
\begin{equation}\label{3d}
2+3t < 2h+4s+6e+6k.
\end{equation}
\item From (\ref{3b}) and (\ref{3d}) we deduce that
\begin{equation}\label{3f}
2+3t< 8e+8k+2r+8s.
\end{equation}
\end{itemize}
  
Next, we consider separately the different finite groups.  

\subsection{Alternate group of degree $4$}

Let $\langle i, \alpha\rangle \cong A_4$. Then, $o(i\alpha)=3$, so $\nu(i\alpha)=\nu((i\alpha)^2)$ and hence $k=0$. Since $i\alpha$ is conjugated to $\alpha$, then $r=s$ and since $i\alpha^2i\alpha$ is conjugated to $i$,  then $e=t$. Therefore, the conditions $g\geq 2$ and $\tilde{B}^2>0$ translate into
\begin{equation}\label{condA4}
h-s\geq 4\hspace{5pt}\textrm{ and }h\leq 3s+t.
\end{equation}

Now, we consider the action of $A_4$ on a curve $D$. The group $A_4$ has three non identity conjugacy classes, those of $i$, $\alpha$ and $\alpha^2$. Since $\alpha$ and $\alpha^2$ have the same fixed points, the Riemann-Hurwitz formula for $D\rightarrow D/A_4$ reads: 
\[
2h-2=24\gamma -24+8s+3t.
\]
Therefore, imposing the lefthand side of \eqref{condA4} we obtain that 
\begin{equation}\label{B0A4}
24\gamma+2s+t\leq 22 \Rightarrow\ \gamma=0
\Rightarrow\ h=4s+\frac{3}{2}t-11.
\end{equation}
By (\ref{B0A4}), the conditions in \eqref{condA4} translate into  
\begin{equation}\label{A4g}
3s+\frac{3}{2}t\geq 15 \text{ and }s+\frac{1}{2}t\leq 11.
\end{equation}

We are going to analyse all possible values of $s$ and $t$ satisfying the inequalities in (\ref{A4g}). With this conditions we observe that we can discard the following cases:
\[
(s,t)\in \{(2,0), (2,2), (2,4), (3,0), (3,2), (4,0)\}.
\]

Given a pair $(s,t)$ satisfying all the conditions, we find a curve $D$ with the action of $A_4$ with the prescribed ramification, when possible, giving the generating vector of $A_4$. If one of the conditions is not satisfied, then there is no such action. These elements determine the branching data for the covering $D\rightarrow D/A_4$ in the following way: there is one branch point for each element, and the monodromy over this branch point is determined by the conjugacy class of the element.
 
According to this, if $s=0$, then the only possible elements in the set of generators are $i$ and its conjugates, that do not generate $A_4$. 

Moreover, if $s=1$, then in any possible set of elements of $A_4$ used to describe the action, there would be one element conjugated to $\alpha$ and the rest would be conjugated to $i$. We observe that all elements conjugated to $i$ have zero or three copies of $\alpha$ on their expression, and all conjugates of $\alpha$ have one, two or four copies of $\alpha$ on their expression. Thus, we deduce that the product of all of them will have $3j\pm 1$ copies of $\alpha$ on the expression, and hence, the condition of product one is not possible to be satisfied. 

For the rest of values satisfying $s+\frac{1}{2}t\leq 11$ we can find a curve with the described action of $A_4$. We study each value of $s$ separately, giving the value of $t$ and the invariants of the curves $C$ and $\tilde{B}\subset C^{(2)}$.  

We list the possible values of $s$ and $t$ and give a table describing the different cases.

{
\setlength{\extrarowheight}{2pt}
\begin{table}[H]\centering 
\begin{tabular}{| c | c | c | c | c | c | c | c |}
\hline
$s$ & $t$ & branching data & $h$ & $g$& $b$& $p_a(\tilde{B})$ & $\tilde{B}^2$ \\
\hline
$2$ & $6$ & $[i]^2[\alpha i \alpha^2][\alpha i \alpha][\alpha]=1$ & $6$& $2$& $2$& $5$& $7$\\
\hline						
 & $8$ & $[i]^3[\alpha i \alpha^2][\alpha i \alpha][\alpha i]=1$ & $9$& $3$& $3$& $7$& $6$\\
\hline
 & $10$ & $[i]^4[\alpha i \alpha^2][\alpha i \alpha][\alpha]=1$ & $12$ &$4$ &$4$& $9$& $5$\\
\hline
 & $12$ & $[i]^5[\alpha i \alpha^2][\alpha i \alpha][\alpha i]=1$ &$15$& $5$ &$5$& $11$& $4$\\
\hline
 & $14$ & $[i]^6[\alpha i \alpha^2][\alpha i \alpha][\alpha]=1$ & $18$& $6$& $6$& $13$& $3$\\
\hline
 & $16$ & $[i]^7[\alpha i \alpha^2][\alpha i \alpha][\alpha i]=1$ &$21$& $7$& $7$ &$15$& $2$\\
\hline
 & $18$ & $[i]^8[\alpha i \alpha^2][\alpha i \alpha][\alpha]=1$ & $24$& $8$& $8$& $17$& $1$\\
\hline

$3$ & $4$ & $[\alpha]^3[i]^2=1$ & $7$& $2$& $3$& $5$& $7$\\
\hline
 & $6$ & $[i\alpha][\alpha^2 i \alpha][\alpha]^2[i]^2=1$ &	$10$& $3$& $4$& $7$& $6$\\
\hline
 & $8$& $[\alpha]^3[i]^4=1$ & $13$& $4$& $5$& $9$& $5$\\
\hline
 & $10$& $[i\alpha][\alpha^2 i \alpha][\alpha]^2[i]^4=1$ &$16$& $5$& $6$& $11$& $4$\\
\hline 
 & $12$ & $[\alpha]^3[i]^6=1$ & $19$& $6$& $7$& $13$& $3$\\
\hline
 & $14$ & $[i\alpha][\alpha^2 i \alpha][\alpha]^2[i]^6=1$ &$22$& $7$& $8$& $15$& $2$\\
\hline
 & $16$ & $[\alpha]^3[i]^8=1$ & $25$& $8$& $9$& $17$& $1$\\
\hline

$4$ & $2$ & $[\alpha i \alpha][\alpha][\alpha i \alpha][\alpha i][i]=1$ &	$8$& $2$& $4$& $5$& $7$\\
\hline
 & $4$ & $[\alpha i\alpha][\alpha][\alpha i \alpha][\alpha][i]^2=1$&	$11$& $3$& $5$& $7$& $6$\\
\hline 
 & $6$& $[\alpha i \alpha][\alpha][\alpha i \alpha][\alpha i][i]^3=1$ & $14$& $4$& $6$& $9$& $5$\\
\hline
 & $8$ & $[\alpha i\alpha][\alpha][\alpha i \alpha][\alpha][i]^4=1$ & $17$& $5$& $7$& $11$& $4$\\
\hline
 & $10$ & $[\alpha i \alpha][\alpha][\alpha i \alpha][\alpha i][i]^5=1$ &$20$& $6$& $8$& $13$& $3$\\
\hline
 & $12$& $[\alpha i\alpha][\alpha][\alpha i \alpha][\alpha][i]^6=1$&	$23$& $7$& $9$& $15$& $2$\\
\hline
 & $14$& $[\alpha i \alpha][\alpha][\alpha i \alpha][\alpha i][i]^7=1$& $26$& $8$ &$10$& $17$& $1$\\
\hline
\end{tabular}
\caption{Action of $A_4$ with $\boldsymbol{s=2,3,4}$}
\label{A4t3}
\end{table}

\begin{table}[H]\centering 
\begin{tabular}{| c |c | c | c | c | c | c | c |}
\hline
$s$ & $t$ & branching data & $h$ & $g$& $b$& $p_a(\tilde{B})$ & $\tilde{B}^2$ \\
\hline
$5$ & $0$ & $[\alpha i \alpha][\alpha][\alpha i \alpha][\alpha^2]^2=1$& $9$& $2$& $5$& $5$& $7$\\
\hline
 & $2$& $[\alpha i\alpha][\alpha][\alpha i \alpha][\alpha^2][\alpha^2 i][i]=1$& $12$& $3$& $6$& $7$& $6$\\
\hline
 & $4$& $[\alpha i \alpha][\alpha][\alpha i \alpha][\alpha^2]^2[i]^2=1$&	$15$& $4$& $7$& $9$& $5$\\
\hline
 & $6$ & $[\alpha i\alpha][\alpha][\alpha i \alpha][\alpha^2][\alpha^2 i][i]^3=1$&	$18$& $5$& $8$& $11$& $4$\\
\hline
 & $8$ & $[\alpha i \alpha][\alpha][\alpha i \alpha][\alpha^2]^2[i]^4=1$& $21$& $6$& $9$& $13$& $3$\\
\hline
 & $10$ & $[\alpha i\alpha][\alpha][\alpha i \alpha][\alpha^2][\alpha^2 i][i]^5=1$ & $24$& $7$& $10$& $15$& $2$\\
\hline
 & $12$ & $[\alpha i \alpha][\alpha][\alpha i \alpha][\alpha^2]^2[i]^6=1$ & $27$& $8$& $11$& $17$& $1$\\
\hline
					
$6$ & $0$ & $[\alpha]^3[\alpha i]^3=1$ & $13$& $3$& $7$& $7$& $6$\\
\hline
 & $2$ & $[i\alpha][\alpha^2 i \alpha][\alpha]^5=1$ & $16$ &$4$& $8$& $9$& $5$\\
\hline
 & $4$ & $[\alpha]^6[i]^2=1$ & $19$& $5$& $9$& $11$& $4$\\
\hline
 & $6$ & $[i\alpha][\alpha^2 i \alpha][\alpha]^5[i]^2=1$ &	$22$& $6$& $10$& $13$& $3$\\
\hline
 & $8$ & $[\alpha]^6[i]^4=1$ & $25$& $7$& $11$& $15$& $2$\\
\hline
 & $10$ & $[i\alpha][\alpha^2 i \alpha][\alpha]^5[i]^4=1$ & $28$& $8$& $12$& $17$& $1$\\
\hline


$7$ & $0$ & $[\alpha]^3[\alpha i \alpha][\alpha][\alpha i \alpha][\alpha]=1$ & 	$13$ &$3$& $7$& $7$& $6$\\
\hline
 & $2$ & $[\alpha]^3[\alpha i \alpha][\alpha][\alpha i \alpha][\alpha i][i]=1$ & $20$& $5$& $10$& $11$& $4$\\
\hline
 & $4$ & $[\alpha]^3[\alpha i \alpha][\alpha][\alpha i \alpha][\alpha][i]^2=1$ & $23$& $6$& $11$& $13$& $3$\\
\hline
 & $6$ & $[\alpha]^3[\alpha i \alpha][\alpha][\alpha i \alpha][\alpha i][i]^3=1$ &	$26$& $7$& $12$& $15$& $2$\\
\hline
 & $8$ & $[\alpha]^3[\alpha i \alpha][\alpha][\alpha i \alpha][\alpha][i]^4=1$ & $29$& $8$& $13$& $17$& $1$\\
\hline

$8$ & $0$ & $[\alpha]^3[\alpha i \alpha][\alpha][\alpha i \alpha][\alpha^2]^2=1$ & $21$& $5$& $11$& $11$& $4$\\
\hline
 & $2$ & $[\alpha]^3[\alpha i \alpha][\alpha][\alpha i \alpha][\alpha^2][\alpha^2 i][i]=1$ & $24$& $6$& $12$& $13$& $3$\\
\hline
 & $4$ & $[\alpha]^3[\alpha i \alpha][\alpha][\alpha i \alpha][\alpha^2]^2[i]^2=1$ & $27$& $7$& $13$& $15$& $2$\\
\hline
 & $6$ & $[\alpha]^3[\alpha i \alpha][\alpha][\alpha i \alpha][\alpha^2][\alpha^2 i][i]^3=1$ & $30$& $8$& $14$& $17$& $1$\\
\hline

$9$ & $0$ & $[\alpha]^6[\alpha i]^3$ & $25$& $6$& $13$& $13$& $3$\\
\hline
 & $2$ & $[i \alpha][\alpha^2 i \alpha][\alpha]^8=1$ & $28$& $7$& $14$& $15$& $2$\\
\hline
 & $4$ & $[\alpha]^9[i]^2$ & $31$& $8$& $15$& $17$& $1$\\
\hline

$10$ & $0$ & $[\alpha]^6[\alpha i\alpha][\alpha][\alpha i \alpha][\alpha]=1$ & $29$& $7$& $15$& $15$& $2$\\
\hline
 & $2$ & $[\alpha]^6[\alpha i\alpha][\alpha][\alpha i \alpha][\alpha i][i]=1$ & $32$& $8$& $16$& $17$& $1$\\
\hline

$11$ & $0$ & $[\alpha]^6[\alpha i\alpha][\alpha][\alpha i \alpha][\alpha^2]^2=1$ & $33$& $8$& $17$& $17$& $1$\\
\hline
\end{tabular}
\caption{Action of $A_4$ with $\boldsymbol{s=5,\dots,11}$}
\label{A4t10}
\end{table}
}

\subsection{Symmetric group of degree $4$}

Let $\langle i, \alpha\rangle \cong S_4$, we recall that $|S_4|=24$.

We take $i=(1\ 2)$, $\alpha=(1\ 4\ 3)$, $i\alpha=(1\ 4\ 3\ 2)$, $(i\alpha)^2=(1\ 3)(2\ 4)$ and $i\alpha^2i\alpha=(1\ 2\ 3)$. Then, $\alpha$ and $i\alpha^2i\alpha$ are conjugated, and we deduce that $s=e$. 

Let $\gamma=g(D/S_4)$ and consider the morphism $D\rightarrow D/S_4$. By Lemma \ref{orbit} we obtain that
the points fixed by $i$ give $\frac{t}{2}$ branch points of order $2$, hence $t$ is even. Similarly, $s$ and $r$  are even and $k$ is multiple of $4$. Therefore, the conditions $g\geq 2$ and $\tilde{B}^2>0$ translate into 
\begin{equation}\label{condS4}
h-s\geq 4\text{ and }h\leq 3s+r+k.
\end{equation}

Now, we consider the action of $S_4$ on the curve $D$. The group $S_4$ has four non identity conjugacy classes, those of $i$ (one transposition), $\alpha$ (cycles of order three), $\beta=i\alpha$ (cycles of order $4$) and $\beta^2$ (double transpositions). Since $\beta$ and $\beta^2$ have the same fixed points, the Riemann-Hurwitz formula for $D\rightarrow D/S_4$ reads:
\[
h=24\gamma -23+3t+4s+\frac{9}{2}r+\frac{3}{2}k.
\]
Therefore, by the left hand side of \eqref{condS4} we obtain that 
\begin{align}
24\gamma +3t+s+\frac{7}{2}r+\frac{1}{2}k\leq 23 \Rightarrow \gamma=0\nonumber \\
\Rightarrow h=-23+3t+4s+\frac{9}{2}r+\frac{3}{2}k.\label{B0S4}
\end{align}
By (\ref{B0S4}), the conditions \eqref{condS4} translate into 
\begin{equation}\label{S4B}
3t+s+\frac{7}{2}r+\frac{1}{2}k\leq 23 \text{ and }
3t+3s+\frac{9}{2}r+\frac{3}{2}k\geq 27.
\end{equation}

Now, we analyse all possible values for $r$, $t$, $s$ and $k$ satisfying the inequalities in (\ref{S4B}). With these conditions we can discard the following cases:
\begin{align*}
(r,t,s,k)\in \{&(0,4,0,0), (0,4,0,4), (0,4,0,8), (0,4,2,0), (0,4,2,4),\\
& (0,4,4,0),(2,2,0,0), (2,2,0,4), (2,2,2,0) \}
\end{align*}

Given $(r,t,s,k)$ satisfying all the conditions, we define a curve $D$ with the action of $S_4$, with the prescribed ramification, giving the generating vector of $S_4$.

For $r=0$ and $t=0$ in any possible set of elements of $S_4$ used to describe the action there would be only elements conjugated to $\alpha$ and $\beta^2$ that have even index, and so they could not generate the whole $S_4$, where there are also odd index elements. 

Since $i$ and $\beta$ have odd index, in order to have the product one condition we need $\frac{t}{2}+\frac{r}{2}$ to be even, or which is the same, $t+r$ to be multiple of four. With this condition we discard the cases   
\[
(r,t)\in\{(0,2), (0,6), (2,0), (2,4), (4,2), (6,0)\}
\]

Since the product of two double transpositions is again a double transposition or $1$, with a detailed study of the multiplication table of $S_4$ we deduce that we cannot have both the generation and the product one condition at the same time in the following cases.
\begin{align*}
(r,t,s,k)\in\{& (0,4,0, 12), (0,4,0,16), (0,4,0,20), (2,2,0,8), (2,2,0,12),\\
& (2,2,0,16), (2,2,0,20), (4,0,0,0), (4,0,0,4)\}
\end{align*}

Hence, there are only two possible pairs of values $(r,t)$. For the remaining values satisfying $3t+s+\frac{7}{2}r+\frac{1}{2}k\leq 23$ we can find a curve with the action of $S_4$. We consider the two pairs of values $(r,t)$ separately, and give a table with the values of $s$ and $k$, the product one relation and the invariants of the curves $C$ and $\tilde{B}\subset C^{(2)}$.  

{\setlength{\extrarowheight}{1.5pt}
\begin{table}[H]\centering 
\begin{tabular}{|c |c | c | c | c | c | c | c |}
\hline
$s$ & $k$ & branching data & $h$ & $g$& $b$& $p_a(\tilde{B})$ & $\tilde{B}^2$ \\
\hline
$2$ &  $8$ & $\begin{array}{c} [(12)][(23)][(132)][(13)(24)]^2= \\ {[t_1]}[t_4][\alpha_1^2][\beta_1^2]^2=1 \end{array}$ (\ref{n1}) & $9$ & $3$& $4$& $9$& $6$\\
\hline
 & $12$ & $\begin{array}{c} [(12)][(23)][(132)][(13)(24)] \\ {[(12)(34)][(14)(23)]=} \\ {[t_1]}[t_4][\alpha_1^2][\beta_1^2][\beta_3^2][\beta_2^2]=1\end{array}$ (\ref{n1}) & $15$ & $5$& $7$& $14$& $4$\\
\hline
& $16$ & $\begin{array}{c} [(12)][(23)][(132)][(13)(24)]^4= \\ {[t_1]}[t_4][\alpha_1^2][\beta_1^2]^4=1\end{array}$ (\ref{n1}) & $21$& $7$& $10$& $19$& $2$\\
\hline
$4$ & $4$ & $[i][\alpha][i][\alpha][(i\alpha)^2]=1$ & $11$& $3$& $5$& $9$& $6$\\
\hline
 & $8$ & $\begin{array}{c} [(23)][(34)][(234)]^2[(12)(34)]^2=\\ {[t_4]}[t_6][\alpha_2]^2[\beta_3^2]=1\end{array}$ (\ref{n3}) & $17$& $5$& $8$& $14$& $4$\\
\hline
 & $12$ & $[i][\alpha][i][\alpha][(i\alpha)^2]^3=1$ & $23$& $7$& $11$& $19$& $2$\\
\hline
$6$ & $0$ & $[i]^2[\alpha]^3=1$ & $13$& $3$& $6$& $9$& $6$\\
\hline
 & $4$ & $\begin{array}{c} [(12)][(34)][(12)(34)][(234)]^3=\\ {[t_1]}[t_6][\beta_3^2][\alpha_2]^3=1 \end{array} $(\ref{n3}) & $19$& $5$& $9$& $14$& $4$\\
\hline
 & $8$ & $[i]^2[\alpha]^3[(i\alpha)^2]^2=1$ & $25$& $7$& $12$& $19$& $2$\\
\hline
$8$ & $0$ & $\begin{array}{c} [(234)]^3[(132)][(13)][(12)]=\\ {[\alpha_2]^3}[\alpha_1^2][t_2][t_1]=1\end{array}$ & $21$& $5$& $10$& $14$& $4$\\
\hline
 & $4$ & $\begin{array}{c}[(243)]^4[(12)][(13)][(12)(34)]=\\ {[\alpha_2^2]}^4[t_1][t_2][\beta_3^2]=1\end{array}$ & $27$& $7$& $13$& $19$& $2$\\
\hline
$10$ & $0$ & $\begin{array}{c} [(234)]^3[(123)]^2[(13)][(12)]=\\ {[\alpha_2]^3}[\alpha_1]^2[t_2][t_1]=1\end{array}$ & $29$& $7$& $14$& $19$& $2$\\
\hline
\end{tabular}
\caption{Action of $S_4$ with $\boldsymbol{r=0}$, $\boldsymbol{t=4}$}
\label{S4t1}
\end{table}
}

\textbf{Notes:}
\begin{enumerate}
\item \label{n1} They generate since $\langle \alpha_1, \beta_1^2\rangle \cong A_4$ and there is no single transposition in this subgroup, so $\langle t_1, \alpha_1, \beta_1^2\rangle \cong S_4$.
\item \label{n3} They generate since $\langle (12), (234) \rangle \cong S_4$.
\end{enumerate}

{\setlength{\extrarowheight}{2pt}
\begin{table}[H]\centering 
\begin{tabular}{|c |c | c | c | c | c | c | c |}
\hline
$s$ & $k$ & branching data & $h$ & $g$& $b$& $p_a(\tilde{B})$ & $\tilde{B}^2$ \\
\hline
$2$ & $4$ & $[i\alpha][i][\alpha][(i\alpha)^2]=1$ & $6$& $2$& $3$& $6$& $7$\\
\hline
 & $8$ & $[i][i\alpha][\alpha^2][(i\alpha)^2]^2=1$ & $12$& $4$& $6$& $11$& $5$\\
\hline
 & $12$ & $[i\alpha][i][\alpha][(i\alpha)^2]^3=1$ & $18$& $6$& $9$& $16$& $3$\\
\hline
 & $16$ & $[i][i\alpha][\alpha^2][(i\alpha)^2]^4=1$& $24$& $8$& $12$& $21$& $1$\\
\hline
$4$ & $0$ & $[i][i\alpha][\alpha]^2=1$& $8$& $2$& $4$& $6$& $7$\\
\hline
 & $4$ & $[(i\alpha)^2][i\alpha][i][\alpha^2]^2=1$ & $14$& $4$& $7$& $11$& $5$\\
\hline
 & $8$& $[i][i\alpha][\alpha]^2[(i\alpha)^2]^2=1$ & $20$& $6$& $10$& $16$& $3$\\
\hline
 & $12$ & $[(i\alpha)^2][i\alpha][i][\alpha^2]^2[(i\alpha)^2]^2=1$& $26$& $8$& $13$& $21$& $1$\\
\hline
$6$& $0$ & $[i][i\alpha][\alpha][\alpha^2]^2=1$& $16$& $4$& $8$& $11$& $5$\\
\hline
 & $4$ & $[(i\alpha)^2][i\alpha][i][\alpha]^2[\alpha^2]=1$ & $22$& $6$& $11$& $16$& $3$\\
\hline
 & $8$ & $[(i\alpha)^2]^2[i][i\alpha][\alpha][(\alpha)^2]^2=1$ & $28$& $8$& $14$& $21$& $1$\\
\hline
$8$ & $0$ & $[i][i\alpha][\alpha]^3[\alpha^2]=1$ & $24$& $6$& $12$& $16$& $3$\\
\hline
 & $4$ & $[(i\alpha)^2][i\alpha][i][\alpha]^4=1$ & $30$& $8$& $15$& $21$& $1$\\
\hline
$10$ & $0$ & $[i][i\alpha][\alpha]^5=1$ & $32$& $8$& $16$& $21$& $1$\\
\hline
\end{tabular}
\caption{Action of $S_4$ with $\boldsymbol{r=2}$, $\boldsymbol{t=2}$}
\label{S4t2}
\end{table}
}

\subsection{Alternate group of degree $5$}

Let $\langle i, \alpha\rangle \cong A_5$, we recall that $|A_5|=60$.

We take $A_5$ embedded in $S_5$ as $i=(1\ 2)(4\ 5)$, $\alpha=(1\ 4\ 3)$ and $i\alpha=(1\ 5\ 4\ 3\ 2)$. Since $o(i\alpha)=5$, a prime number, we deduce that $i\alpha$ and $(i\alpha)^2$ have the same fixed points and therefore $k=0$. Moreover, $i\alpha^2i\alpha=(1\ 4\ 5\ 2\ 3)$ that is conjugated of $i\alpha$ or $(i\alpha)^2$, so $e=r$. 

Let $\gamma=g(D/A_5)$ and consider the morphism $D\rightarrow D/A_5$. By Lemma \ref{orbit} we obtain that
the points fixed by $i$ give $\frac{t}{2}$ branch points of order $2$, hence $t$ is even. Similarly, $s$ and $r$ are even. Therefore, the conditions $g\geq 2$ and $\tilde{B}^2>0$ translate into
\begin{equation}\label{condA5}
h-s\geq 4\text{ and }h\leq 2s+2r.
\end{equation}

Now, we consider the action of $A_5$ on the curve $D$. The group $A_5$ has four non identity conjugacy classes, those of $i$ (double transpositions), $\alpha$ (cycles of order three), $\beta=i\alpha$ (cycles of order $5$) and $\beta^2$, nevertheless, since $\beta$ and $\beta^2$ have the same fixed points, the Riemann-Hurwitz formula for $D\rightarrow D/A_5$ reads:
\[
 h=60\gamma -59+\frac{15}{2}t+10s+12r.
\]
Therefore, by the left hand side of \eqref{condA5} we obtain that 
\begin{align}
60\gamma +\frac{15}{2}t+8s+10r\leq 59 \Rightarrow \gamma=0\nonumber \\
\Rightarrow h=-59+\frac{15}{2}t+10s+12r.\label{B0A5}
\end{align}

By (\ref{B0A5}), the conditions in \eqref{condA5} translate into 
\begin{equation}\label{A5B}
\frac{15}{2}t+8s+10r\leq 59 \text{ and }
\frac{15}{2}t+9s+12r\geq 63.
\end{equation}

We are going to analyse all possible values for $r$, $s$ and $t$ satisfying the inequalities (\ref{A5B}). With this conditions we observe that we can discard the following cases $r=0$, $r=2$ and $(r,s,t)=(4,0,0)$

It remains only the possibility $r=4$. We describe in the following table the possible actions of $A_5$ on a curve $D$ with the ramification determined by the values of $r$, $s$ and $t$. 

We give the value of $s$ and $t$, the product one relation and the invariants of the curves $C$ and $\tilde{B}\subset C^{(2)}$. 

{\setlength{\extrarowheight}{2pt}
\begin{table}[H]\centering 
\begin{tabular}{|c |c | c | c | c | c | c | c |}
\hline
$s$ & $t$ & branching data & $h$ & $g$& $b$& $p_a(\tilde{B})$ & $\tilde{B}^2$ \\
\hline
$0$ & $2$ & $[(12345)][(13)(24)][(15234)]=1$ (\ref{n51}) & $4$& $2$& $2$& $4$& $5$\\
\hline
$2$ & $0$ & $[i \alpha][\alpha][\alpha i]=1$ & $9$& $3$& $5$& $7$& $4$\\
\hline
\end{tabular}
\caption{Action of $A_5$ with $\boldsymbol{r=4}$}
\label{A5t1}
\end{table}
}

\textbf{Note:}
\begin{enumerate}
\item \label{n51} These three elements generate $A_5$ because in $A_5$ an element of order two and one of order five can only generate $D_5$ or $A_5$, but since $(15234)\neq (12345)^j$ it cannot be $D_5$ (since all elements of order five in $D_5$ are a cyclic group).
\end{enumerate}

With this we finish the study of degree three curves with positive self-intersection defined by the action of a spherical triangular group on a curve $D$.


\bibliographystyle{alpha}
\bibliography{bib}

\end{document}